\documentclass[12pt]{amsart}
\usepackage{amssymb,latexsym, amsmath, amscd, array, graphicx}

\swapnumbers \numberwithin{equation}{section}

\theoremstyle{plain}

\newtheorem{thm}{Theorem}[section]

\newtheorem{lem}[thm]{Lemma}
\newtheorem{lemma}[thm]{Lemma}

\newtheorem{prop}[thm]{Proposition}
\newtheorem{cor}[thm]{Corollary}

\theoremstyle{definition}
\newtheorem{defin}[thm]{Definition}

\newtheorem{remark}[thm]{Remark}

\newtheorem{ex}[thm]{Example}

\newtheorem{question}[thm]{Question}

\DeclareMathOperator{\cat}{{\mbox{\rm cat$_{\rm LS}$}}}

\DeclareMathOperator{\cuplength}{{\rm cup-length}}

\DeclareMathOperator{\Ord}{{\rm Ord}}

\def\cat{\protect\operatorname{cat}}

\def\Z{{\mathbb Z}}

\def\1{\hbox{\rm\rlap {1}\hskip.03in{\rom I}}}
\def\Bbbone{{\rm1\mathchoice{\kern-0.25em}{\kern-0.25em}
{\kern-0.2em}{\kern-0.2em}I}}

\long\def\forget#1\forgotten{} %

\newcommand\ver[1]{\marginpar{\tiny Changed in Ver \VER}}

\date{\today}
\begin{document}

\title[On LS-category of Peano continua]{On the Lusternik-Schnirelmann category of Peano continua}

\author[Tulsi Srinivasan]{Tulsi Srinivasan} %

\address{Tulsi Srinivasan, Department of Mathematics, University
of Florida, 358 Little Hall, Gainesville, FL 32611-8105, USA}
\email{tsrinivasan@ufl.edu}


\begin{abstract}
We define the LS-category $\cat_g$ by means of covers of a space by
general subsets, and show that this definition coincides with the classical
Lusternik-Schnirelmann category for compact metric ANR spaces. We apply this result
to give short dimension theoretic proofs of the Grossman-Whitehead theorem and
Dranishnikov's theorem. We compute $\cat_g$ for some fractal
Peano continua such as Menger spaces and Pontryagin surfaces.
\end{abstract}

\maketitle

\section{Introduction}

We recall that the {\em Lusternik-Schnirelmann category} $\cat X$ of
a topological space $X$ is the smallest integer $k$ such that $X =
\bigcup_{i=0}^{k}A_{i}$, where each $A_{i}$ is an open set
contractible in $X$. A set $A\subset X$ is said to be {\em
contractible in $X$} if its inclusion map $A\to X$ is homotopic to
the constant map. It is known that for absolute neighborhood
retracts (ANR spaces) the sets $A_i$ can be taken to be closed. In
this paper we investigate what happens if the $A_i$ are arbitrary subsets.
Thus, we consider the following:

\begin{defin} \label{def} For any space $X$, define
{\em the general LS-category} $\cat_{g}X$ of $X$ to
be the smallest integer $k$ such that $X = \bigcup_{i=0}^{k}A_{i}$,
where each $A_{i}$ is contractible in $X$.
\end{defin}

We apply this definition to Peano continua, i.e., path connected and
locally path connected compact metric spaces. On these spaces, $\cat_g$
has the same upper and lower bounds as $\cat$, i.e., $$
\cuplength(X)\le\cat_gX\le\dim X.
$$

Moreover, the Grossman-Whitehead theorem (Corollary ~\ref{Whitehead}) and Dranishnikov's theorem (Corollary ~\ref{Dranish}) can be proven
for $\cat_g$, using dimension theoretic arguments.
This allows us to compute $\cat_g$ for fractal spaces
such as the Sierpinski carpet, Menger spaces and
Pontryagin surfaces.

We show that $\cat_g X=\cat X$ for compact metric ANRs, which yields new short proofs of 
the Grossman-Whitehead theorem and Dranishnikov's theorem.

\section{Acknowledgement}
I am very grateful to my adviser Alexander Dranishnikov for
formulating this problem, and for all his ideas and advice.

\section{LS-category for general spaces}

The following result is known.

\begin{prop}\label{nerve}
Let $X$ be a metric space, $A$ a subset of $X$ and
$\mathcal V' = \{V'_i\}_{i\in I}$ a cover of $A$ by sets open in $A$. Then
$\mathcal V'$ can be extended to a cover $\mathcal V = \{V_i\}_{i\in I}$ of
$A$ by sets open in $X$ with the same nerve and such that $V_i\cap A=V_i'$ for all $i\in I$.

\begin{proof} Let $$V_i = \bigcup_{a \in V'_i}B(a, d(a, A - V'_i)/2),$$
and let $\mathcal V = \{V_i\}$. Clearly $\mathcal V$ is an extension
of $\mathcal V'$ and  $V_i\cap A=V_i'$. We claim that $V_{i_1}\cap\dots\cap V_{i_k}\ne\emptyset$
if and only if $V_{i_1}'\cap\dots\cap V_{i_k}'\ne\emptyset$. We prove the claim only for $k=2$, since 
we apply it only in this case. A similar proof holds for $k > 2$. 

Thus, we show that if $V_i\cap V_j\ne\emptyset$, then
$V'_i\cap V_j'\ne\emptyset$.
If $x \in V_i \cap V_j$, then there exist $a_i \in V'_i$ and $a_j
\in V'_j$ for which $d(x,a_i) < d(a_i, A - V'_i)/2$ and $d(x,a_j) <
d(a_j, A - V'_j)/2$. Suppose $d(a_i, A - V'_i) = \max\{d(a_i, A -
V'_i), d(a_j, A - V'_j)\}$. Then $d(a_i,a_j) < d(a_i, A - V'_i)$, so
$a_j \in V'_i$, which means that $a_j \in V'_i \cap V'_j $.
\end{proof}
\end{prop}

We refer to ~\cite{Hu} for the definitions of absolute neighborhood
retracts (ANRs) and absolute neighborhood extensors (ANEs).

\begin{thm} ~\cite{Hu}
A metrizable space $X$ is an ANR for
metrizable spaces iff it is an ANE
for metrizable spaces. \label{ANE}
\end{thm}

The following is a version of a lemma that appears in ~\cite{Wa}.

\begin{thm}[Walsh Lemma] \label{Walsh}
Let $X$ be a separable metric space, $A$ a subset of $X$, $K$ a metric separable ANR,
and $f: A \rightarrow K$ a map. Then, for any $\epsilon > 0$, there
is an open set $U \supset A$ and a map $g: U \rightarrow K$ such
that:
\begin{enumerate}
    \item $g(U)$ is contained in an $\epsilon$-neighborhood of $f(A)$
    \item $g|_A$ is homotopic
to $f$.
\end{enumerate}

\begin{proof} We use the fact that any Polish space is homeomorphic
to a closed subspace of a Hilbert space H ~\cite{Ch}. Since $K$ is an ANR,
there is an open neighborhood $O$ of $K$ in $H$, and a retraction
$r:O \rightarrow K$. For every $y \in K$,
pick $\delta_{y} > 0$ such that:
\begin{enumerate}
    \item [(i)]$B(y, 2\delta_{y}) \subset O$,
    \item [(ii)] For all $y_1,y_2 \in B(y, 2\delta_{y})$, we have $d(r(y_1),r(y_2)) < \epsilon$ .
\end{enumerate}

For each $a \in A$, pick a neighborhood $U_a$ of $a$ that is open in $A$ so that
$f(U_a) \subset B(f(a),\delta_{f(a)})$. Let $\mathcal{V'} =
\{V'_i\}$ be a locally finite refinement of the collection of $U_a$, and
$\mathcal{V} = \{V_i\}$ the cover obtained by applying Proposition ~\ref{nerve} to $\mathcal V'$. Let $U = \bigcup_iV_i$.

For each $i$, fix $a_i \in V'_i$. Let $\{f_{i}: V_i \in \mathcal{V}\}$ be a partition of unity
subordinate to  $\mathcal{V}$. Define $h: U \rightarrow H$ by $h(u) = \sum_{i}f_i(u)f(a_i)$.

Choose any $u \in U$. Then $u$ lies in precisely $k$ of the $V_i$,
say in $V_{i_1},...,V_{i_k}$. Assume that $\delta_{f(a_{i_j})} =
\max\{\delta_{f(a_{i_1})},...,\delta_{f(a_{i_k})}\}$. Then all the
$f(a_{i_l})$ lie in $B(f(a_{i_j}),2\delta_{f(a_{i_j})})$, so $h(u)$
lies in this ball too. This means that $h$ is a map from $U$ to $O$.

Define $g: U \rightarrow K$ by $g = r \circ h$. If $u \in U$, then
we have seen that $h(u)$ lies in $B(f(a_u),2\delta_{f(a_u)})$ for
some $a_u \in A$, so $d(g(u),f(a_u)) < \epsilon$, which implies that
$g(U) \subset N_\epsilon(f(A))$. Note here that we could have taken
the $V_i$ to have as small a diameter as required, so the distance
between $u$ and $a_u$ can be made as small as necessary. This will
be made use of in Lemma ~\ref{extn}.

For every $a \in A$, $h(a)$ and $f(a)$ lie in some convex ball, so
$h|_A$ is homotopic to $f$ in $O$ via the straight line homotopy.
The composition of this homotopy with $r$ is then a homotopy in $K$
between $g|_A$ and $f$.

\end{proof}
\end{thm}

\begin{thm}  ~\cite{B}
\label{epsilon} Let $K$ be a compact ANE. Then there exists a constant 
$\epsilon(K) > 0$ such that for any metric space $X$ and maps 
$f,g: X \rightarrow K$, if $d(f(x),g(x)) < \epsilon(K)$ 
for all $x \in X$, then $f$ is homotopic to $g$.
\end{thm}

\begin{prop} Let $K$ be a compact ANE, $y_{0} \in K$ and $PK$ the path space $\{\phi:
[0,1] \rightarrow K|\phi(1) = y_{0}\}$ with sup norm $D$. If $\epsilon(K)$ is as in
{Theorem~\ref{epsilon}}, then any two maps $f,g: X \rightarrow PK$
such that $D(f(x),g(x)) < \epsilon(K)$ for all $x \in X$, are
homotopic to each other.
\begin{proof} Since $PK$ is endowed with the sup norm, we have
$$d(f(x)(t),g(x)(t)) < \epsilon(K)$$ for every $x \in X, t \in [0,1]$.
Define $F: X \times I \rightarrow K$ by $F(x,t) = f(x)(t)$ and $G: X \times I
\rightarrow K$ by $G(x,t) = g(x)(t)$.

Since $d(F(x,t),G(x,t)) < \epsilon(K)$ for all $(x,t) \in X \times
I$, there exists a homotopy $h_{s}: X \times I \rightarrow K$
between $F$ and $G$. Define $\tilde{h}_{s}: X \rightarrow PK$ by
$\tilde{h}_{s}(x)(t) = h_{s}(x,t)$. Then $\tilde{h}_{s}$ is the
required homotopy between $f(x)$ and $g(x)$. \label{AE1}
\end{proof}
\end{prop}

The following is well known.

\begin{prop} If $K$ is an ANE, so is $PK$.
\begin{proof} Suppose $K$ is an ANE. Let $A$ be a closed subspace of
a metric space $X$ and $f:A \rightarrow PK$ a map. Define $F: A
\times I \rightarrow K$ by $F(x,t) = f(x)(t)$. By hypothesis, $F$
extends over an open neighborhood $V$ of $A \times I$ in $X \times
I$. For each $a \in A$, find an open neighborhood $U_{a}$ such that
$U_{a} \times I \subset V$. Let $U = \bigcup_{a \in A}U_{a}$. The
map $\tilde{f}: U \rightarrow PK$ given by $\tilde{f}(x)(t) =
\tilde{F}(x,t)$ is the required extension. \label{AE2}
\end{proof}
\end{prop}

\begin{lemma}\label{extn}
Let $X$ be a compact metric space, $A \subset X$, $K$ a compact
metric ANR, and $f: X \rightarrow K$ a map such that the restriction
of $f$ to $A$ is nullhomotopic. Then there exists $U \supset A$ open
in $X$ such that the restriction of $f$ to $U$ is nullhomotopic.

\begin{proof} Since $K$ is compact, there is an $\epsilon$ such
that any two $\epsilon$-close maps to $K$ are homotopic. Let $D$ be the
metric on $K$ and $d$ the metric on $X$. Let $PK$ be the path space
$\{\phi: [0,1] \rightarrow K| \phi(1) = y_{0}\}$ (for some $y_{0} \in K$),
under the sup metric $D'$. As $f|_{A}$ is nullhomotopic, there is a
map $F: A \rightarrow PK$ satisfying $F(a)(0) = f(a)$ for all $a \in A$.
By the uniform continuity of $f$, there is a $\delta > 0$ such that
$d(x,y) < \delta \Rightarrow D(f(x),f(y)) < \displaystyle\frac{\epsilon}{2}.$

By {Proposition~\ref{AE2}} and Theorem ~\ref{ANE}, $PK$ is an ANR for metric spaces. By
{Proposition~\ref{AE1}}, any two $\epsilon$-close maps to it are
homotopic. We apply Theorem ~\ref{Walsh} to $F$, and construct an open
neighborhood $U$ of $A$ and a map $G: U \rightarrow PK$ such that
for every $u \in U$, there exists $a_u \in A$ such that
$D'(G(u),F(a_u)) < \displaystyle\frac{\epsilon}{2}$. As noted in the proof of
Theorem ~\ref{Walsh}, we can construct $U$ so that diam$V_i < \delta$
for all $i$, so $d(u,a_u) < \delta$. Let $g: U \rightarrow K$ be given by $g(u) =
G(u)(0)$. For any $u \in U$, we have \begin{gather*}D(g(u),f(u))
\leq  D(g(u),f(a_{u})) + D(f(a_{u}),f(u)) \\ <
D(G(u)(0),F(a_{u})(0)) + \displaystyle\frac{\epsilon}{2} \\ \leq
D'(G(u),F(a_{u})) + \displaystyle\frac{\epsilon}{2} \\ <
\displaystyle\frac{\epsilon}{2} + \displaystyle\frac{\epsilon}{2} =
\epsilon. \end{gather*} It follows that $f|_{U}$ is homotopic to
$g$.

But $h_t: U \rightarrow K$ given by $h_t(u) = G(u)(t)$ is a
homotopy between $g$ and the constant map at $y_0$, so $f|_U$ is
nullhomotopic.
\end{proof}
\end{lemma}

\begin{cor} Let $X$ be a compact metric ANR, and let $A \subset X$ be
contractible in $X$. Then there exists an open set $U \supset A$
that is contractible in $X$. \label{contractible}
\end{cor}

\begin{thm} For a compact metric ANR $X$, $\cat_gX = \cat X$.
\begin{proof} Clearly $\cat_gX \leq \cat X$, and equality
holds if $\cat_gX$ is infinite. Suppose $\cat_gX =
n$. Then we can write $X = \bigcup_{i = 0}^{n}A_i$, where the
$A_i$ are contractible in $X$. By {Corollary~\ref{contractible}},
there exist open sets $U_i$ containing $A_i$ that are
contractible in $X$ for each $i$. It follows that $\cat X \leq n.$
\end{proof}
\end{thm}

The following proposition illustrates some of the basic properties of $\cat_g$. The proofs are nearly identical to those in the case of the usual definition of category (see ~\cite{CLOT}), and so are omitted. 

\begin{prop}[Properties of $\cat_g $]\label{properties} If $A,B,X,Y$ are spaces and $f: X \rightarrow Y$ is a map with mapping cone $C_f$, then the following hold:

i) $\cat_g (A \cup B) \leq \cat_g A + \cat_g B + 1$

ii) If $f$ has a right homotopy inverse, then $\cat_g X \geq \cat_g Y$

iii) If $f$ is a homotopy equivalence, then $\cat_g X = \cat_g Y$

iv) $\cat_g (C_f) \leq \cat_g Y + 1$.

\end{prop}

\vspace{2mm} The following proposition shows that the sets in the
definition of the LS-category can be assumed to be
$G_{\delta}$.

\begin{prop} \label{G-delta} Let $X$ be a complete metric space.
If $A \subset X$ is contractible in $X$, then there exists a
$G_{\delta}$ set $B$ with $A \subset B \subset X$ such that $B$
is contractible in $X$.
\end{prop}

The proof is based on the following classical theorem ~\cite{E2}:

\begin{thm}[Lavrentieff] \label{Lavrentieff}If  $Y$ is a complete metric space, then any continuous
map $f: A \rightarrow Y$, where $A$ is a dense subset of a space $X$,
can be extended to $F: B \rightarrow Y$, where $B$ is a $G_{\delta}$
set in $X$ containing $A$.
\end{thm}

\begin{proof}[Proof of Proposition ~\ref{G-delta}]
Consider $A$ as a subspace of its closure $\bar A$. Since $A$ is contractible 
in $X$, there is a map $f: A \rightarrow PX$ satisfying $f(a)(0) = a$. By 
Theorem ~\ref{Lavrentieff}, $f$ can be extended to $F: B \rightarrow PX$, 
where $B$ is a $G_\delta$ set in $\bar A$ containing $A$. Since $B$ is 
clearly also $G_\delta$ in $X$, it only remains to show that it is 
contractible in $X$.

Pick any $b \in B$. Then $b$ is the limit of some sequence 
$\{a_n\}$ in $A$, so $$F(b)(0) = F(\lim_{n\rightarrow \infty}a_n)(0) = 
\lim_{n \rightarrow \infty}F(a_n)(0) = \lim_{n\rightarrow \infty}f(a_n)(0) = 
\lim_{n \rightarrow \infty}a_n = b,$$
so $B$ is contractible in $X$.

\end{proof}

\section{Upper bounds}

We need the following definitions for the Grossman-Whitehead theorem. An {\em
absolute extensor in dimension $k$}, or {\em $AE(k)$}, is a space
$X$ with the property that for every space $Z$ with $\dim Z \leq k$
and closed subset $Y \subset Z$, any map $f: Y \rightarrow X$ can be extended
over all of $Z$. A {\em $k$-connected} or {\em $C^k$} space is a
space whose first $k$ homotopy groups are trivial. A {\em locally
$k$-connected} or {\em $LC^k$} space is a space $X$ with the
property that for every $x \in X$ and neighborhood $U$ of $x$, there
is a neighborhood $V$ with $x \in V \subset U$ such that every map
$f: S^r \rightarrow V$ is nullhomotopic in $U$ for $r\leq k$.

\begin{thm}[Kuratowski] For $k \geq 0$, a metrizable
space $X$ is $AE(k+1)$ for metrizable spaces iff it is $LC^k$ and $C^k$.
\label{Kuratowski}
\end{thm}

We use the notation $CX$ for the cone over $X$, $$CX=(X\times [0,1])/(X\times\{1\}).$$

\begin{thm} For $k \geq 0$, let $X$ be an $LC^{k}$ and $C^{k}$ compactum. Then $\cat_{g}X \leq \dim X/(k+1)$.
\begin{proof} Suppose $\dim X = n = p(k+1) + r$, where $0 \leq r < k+1$. We can write $X = \bigcup_{i = 0}^{p(k+1) + r}X_{i}$, where
$\dim X_{i} = 0$ for each $i$. Let $A_{i} = X_{(k+1)i} \cup X_{(k+1)i+1}
\cup... \cup X_{(k+1)(i+1) - 1}$ for $i = 0,...,p-1$, and let $A_{p} =
X_{p(k+1)} \cup... \cup X_{p(k+1)+r}$. Then $\dim A_{i} \leq k$ for each
$i$, so $\dim CA_{i} \leq k+1.$

For each $i$, consider the inclusion map $A_{i} \hookrightarrow X$.
Since $X$ is $k$-connected and locally $k$-connected, $X$ is an
AE$(k+1)$ space, and so the inclusion maps extend over $CA_{i}$ for
each $i$. Hence we have $X = \bigcup_{i=0}^{p} A_{i}$, where each
$A_{i}$ is contractible in $X$, which implies that $\cat_g X \leq p \leq
\dim X/(k+1)$.
\end{proof}
\end{thm}

\begin{cor}[Grossman-Whitehead Theorem] \label{Whitehead} For a $k$-connected complex $X$,
$$\cat X\le \dim  X/(k+1).$$
\end{cor}

\begin{ex}
For the $n$-dimensional Menger space $\mu^n$,
$$
\cat_g \mu^n=1.
$$
\begin{proof} The $n$-dimensional Menger space is $(n-1)$-connected
and $(n-1)$-locally connected ~\cite{Be}, so $\cat_{g}\mu_n \leq 1$.
Since $\mu^n$ is not contractible, $\cat_g \mu_n = 1$.
\end{proof}
\end{ex}

\begin{defin}
i) A family $\{A_{i}\}$ of subsets of $X$ is called an $n$-cover if every
subfamily of $n$ sets forms a cover: $X=A_{i_1}\cup\dots\cup
A_{i_n}$.

ii) Given an open cover $\mathcal U$ of $X$ and a point $x \in X$, the order of $\mathcal U$ at $x$, $\Ord_x\mathcal U$, is the number of elements of $\mathcal U$ that contain $x$.
\end{defin}

We will make use of the following result that appears in ~\cite{Dr2}:

\begin{prop} \label{order} 
A family $\mathcal U$ that consists of $m$ subsets of $X$ is an $(n + 1)$-cover
of $X$ if and only if $\Ord_x \mathcal U \geq m-n$ for all $x \in X$.
\begin{proof} See ~\cite{Dr1, Dr2}
\end{proof}
\end{prop}

\begin{thm} ~\cite{Os} \label{0-dim}
For every $m>n$, every $n$-dimensional compactum $X$ admits an
$(n+1)$-cover by $m$ $0$-dimensional sets.
\begin{proof} This result follows from a slight modification to a proof given in ~\cite{Os}.

Since $\dim X \leq n$, $X$ can be decomposed into $0$-dimensional
sets as $X = X_0 \cup...\cup X_n$. We assume that the $X_i$
are $G_{\delta}$ sets ~\cite[Theorem 1.2.14]{E}, and proceed inductively. For any $m > n+1$, suppose an $(n+1)$-cover $\{X_{0},...,X_{m-1}\}$ consisting of $0$-dimensional $G_{\delta}$
sets has been constructed. Let $$X_m = \{x \in X| x \mbox{ lies in
exactly }(m-n) \mbox{ of the } X_i\}.$$ 

Since the $X_0,...,X_{m-1}$ form an $(n+1)$-cover, Proposition ~\ref{order} implies that each $x \in X$ lies in at least $m-n$ of the $X_i$. Then $X_m$ is the complement in $X$ of a finite union of finite intersections of $G_\delta$ sets, and is therefore $F_\sigma$. Similarly, for $0 \leq i \leq m-1$, each $X_m \cap X_i$ is a $0$-dimensional set that is $F_{\sigma}$ in $X$, and therefore in $X_m$. As the finite union of $0$-dimensional $F_\sigma$ sets, $X_m$ is also $0$-dimensional ~\cite[Corollary 1.3.3]{E}. It is also clear from the construction of $X_m$ that $\{X_0,...,X_m\}$ is an $(n+1)$-cover of $X$.
\end{proof}
\end{thm}

\begin{cor}\label{1-dim}
For every $m>\left\lfloor n/2\right\rfloor$ every $n$-dimensional
compactum admits an $(\left\lfloor n/2\right\rfloor+1)$-cover by $m$
$1$-dimensional sets.
\begin{proof} Decompose $X$ into $0$-dimensional sets $X_{0},...,X_{n}$
as before, and group these into pairs to get (at most) $1$-dimensional
$G_{\delta}$ sets $Y_0,..,Y_{\left\lfloor n/2\right\rfloor}$ that cover $X$.

We proceed by induction again. For any $m > \left\lfloor n/2\right\rfloor$, suppose an
$(\left\lfloor n/2\right\rfloor+1)$-cover $\{Y_0,..,Y_{m-1}\}$
consisting of $1$-dimensional $G_{\delta}$ sets has been
constructed. Let $Y_m$ be the $F_\sigma$ set $\{x \in X| x \mbox{
lies in exactly }(m-\left\lfloor n/2\right\rfloor ) \mbox{ of the }
Y_i\}$. Each $Y_m \cap Y_i$ is $F_{\sigma}$ of dimension $\leq 1$,
so $\dim Y_m \leq 1$ ~\cite[Theorem 1.5.3]{E}, and $\{Y_0,...,Y_m\}$ is the desired
$(\left\lfloor n/2\right\rfloor+1)$-cover of $X$.
\end{proof}
\end{cor}

The following lemma can be traced back to Kolmogorov (see ~\cite{Os},~\cite[Proof ofTheorem 3.2]{Dr2}).
\begin{lem}\label{KO}
Let $A_0,\dots, A_{m+n}$ be an $(n+1)$-cover of $X$ and $B_0,\dots,
B_{m+n}$ an $(m+1)$-cover of $Y$. Then $A_0\times B_0,\dots,
A_{m+n}\times B_{m+n}$ is a cover of $X\times Y$.
\end{lem}
We recall that the geometric dimension $gd(\pi)$ of a group $\pi$ is
defined as the minimum dimension of all Eilenberg-Maclane complexes
$K(\pi,1)$. It is known that $gd(\pi)$ coincides with the
cohomological dimension of the group, $cd(\pi)$, for all groups with
$gd(\pi)\ne 3$~\cite{Br}. The Eilenberg-Ganea conjecture asserts
that the equality $gd(\pi)=cd(\pi)$ holds true for all groups $\pi$.

The following theorem was proven by Dranishnikov ~\cite{Dr2} for CW
complexes. We present here a new short proof based on his idea to
use the general LS-category.
\begin{thm}
Let $X$ be a semi-locally simply connected Peano continuum. Then
$$
\cat_{g}X\leq gd(\pi_1(X))+\frac{\dim X}{2}.
$$
\end{thm}
\begin{proof}
The conditions imply that $X$ has the universal covering space
$p:\tilde X\to X$. Let $\pi=\pi_1(X)$ and let $q:E\to K(\pi,1)$ be
the universal cover. Note that the orbit space $\tilde X\times_{\pi}E$
under the diagonal action of $\pi$ on $\tilde X\times E$ has the
projections $p_1:\tilde X\times_{\pi}E\to X$ and $p_2:\tilde
X\times_{\pi}E\to K(\pi,1)$ which are locally trivial bundles. Since
$E$ is contractible, $p_1$ admits a section $s$, so by Proposition ~\ref{properties},
$\cat_gX\le\cat_g(\tilde X\times_{\pi}E)$.

The projection $p\times q:\tilde X\times E\to X\times K(\pi,1)$ is
the projection onto the orbit space of the action of the group
$\pi\times\pi$. Therefore, it factors through the projection
$\xi:\tilde X\times E\to \tilde X\times_{\pi}E$ of the orbit action
of the diagonal subgroup $\pi\subset\pi\times\pi$, $p\times
q=\psi\circ\xi$ as follows:
$$
\tilde X\times E\stackrel{\xi}{\longrightarrow} \tilde X\times_{\pi}E\stackrel{\psi}{\longrightarrow} X\times K(\pi,1).
$$

Let $\dim X=n$ and $\dim(K(\pi,1))=gd(\pi)=m$. We apply
{Corollary~\ref{1-dim}} to $X$ and {Theorem~\ref{0-dim}} to
$K(\pi,1)$ to obtain an $(\left\lfloor n/2\right\rfloor+1)$-cover
$A_0,\dots, A_r$ of $X$ by 1-dimensional sets and an $(m+1)$-cover
$B_0,\dots, B_r$ of $K(\pi,1)$ by $0$-dimensional sets, for
$r=m+\left\lfloor n/2\right\rfloor$. By Lemma~\ref{KO}, $A_0\times
B_0,\dots, A_r\times B_r$ is a cover of $X\times K(\pi,1)$ by
$1$-dimensional sets.

If $f: X \rightarrow Y$ is an open surjection between metric separable
spaces such that the fiber $f^{-1}(y)$ is discrete for each $y \in
Y$, then $\dim X = \dim Y$ ~\cite{E}. The map $\psi \circ \xi =
(p,q)$ is an open map, each fiber of which is discrete. An open set
in $\tilde X \times_{\pi} E$ is taken to an open set in $\tilde X
\times E$ by $\xi^{-1}$, which is taken to an open set in $X \times
K(\pi,1)$ by $\psi \circ \xi$, and so, using the surjectivity of
$\xi$, we can say that $\psi$ is open. Similarly, the image under
$\xi$ of any discrete set is still discrete, so $\psi^{-1}(y)$ is
discrete for every $y \in X \times K(\pi,1)$. It follows that
$\{\psi^{-1}(A_i\times B_i)\}_{i=0}^r$ is a cover of $\tilde
X\times_{\pi}E$ by $1$-dimensional sets.

We show that each $\psi^{-1}(A_i\times B_i)$ is contractible in
$\tilde X\times_{\pi}E$. This will imply that $$\cat_g(\tilde
X\times_{\pi}E)\le gd(\pi)+\left\lfloor n/2\right\rfloor \le gd(\pi)
+ \frac{\dim X}{2}.$$ Note that $p_2(\psi^{-1}(A_i\times
B_i))=B_i$. Since $B_i$ is contractible to a point in $K(\pi,1)$,
the set $\psi^{-1}(A_i\times B_i)$ can be homotoped to a fiber
$p_2^{-1}(x_0)\cong\tilde X$. Since $\tilde X$ is a simply
connected and each $\psi^{-1}(A_i \times B_i)$ is $1$-dimensional,
Theorem ~\ref{Kuratowski} implies that the inclusion map of
each subspace can be extended over its cone, and so
each $\psi^{-1}(A_i\times B_i)$ can be contracted to a point
in $\tilde X\times_{\pi}E$ .
\end{proof}
\begin{cor}[Dranishnikov's Theorem]\label{Dranish} For a finite CW complex $X$,
$$
\cat X\leq gd(\pi_1(X))+\frac{\dim X}{2}.
$$
\end{cor}

\section{Lower bounds}

\begin{defin} Let $R$ be a commutative ring. The $R$-cup-length $\cuplength_R X$
of a space $X$ is the smallest integer $k$ such that all cup-products
of length $k+1$ vanish in the \v Cech cohomology ring $\tilde{H}^{*}(X;R)$.
\end{defin}

\begin{thm}\label{cup}
Let $X$ be a compactum with $\cat_g X \leq m$. Then $$\cuplength_R X\le m$$  for any ring  $R$.
\begin{proof} Let $\{A_i\}_{i=0}^{m}$ be as in {Definition ~\ref{def}}.
Assume the contrary, i.e., that there exists a non-zero product
$\alpha_{0} \smile \alpha_{1} \smile...\smile \alpha_{m}$ in
$\tilde{H}^*(X;R)$, where $\alpha_i \in H^{k_i}(X;R)$,
$k_i>0$. For each $i$, there exists a function $f_i: X \rightarrow
K(R,k_i)$ such that $\alpha_{i}$ belongs to the homotopy class
$[f_i]$ (in view of the isomorphism between the group of homotopy
classes $[X,K(R,n)]$ and the \v Cech cohomology group $H^n(X;R)$
~\cite{Sp}, and the fact that the \v Cech cohomology agrees with the
Alexander-Spanier cohomology). Since $X$ is compact, for each $i$
there is a finite subcomplex $K_i\subset K(R,k_i)$ such that
$f_i(X)\subset K_i$. Since each $A_i$ is contractible in $X$,
$f_i|_{A_i}:A_i\to K_i$ is nullhomotopic. By Lemma~\ref{extn}, there exists an open neighborhood $U_i$ of
$A_i$ such that $f_i|_{U_i}:U_i\to K_i$ is also nullhomotopic.

Let $j_i$ be the inclusion $U_i \rightarrow X$, and $q_i$ the
map $X \rightarrow (X,U_i)$. We consider the exact sequence
of the pair $(X,U_i)$ in the Alexander-Spanier  cohomology
(see~\cite{Sp}, p. 308-309):
\begin{gather*}\begin{CD} ...\rightarrow H^{k_{i}}(X,U_{i};R) @>q^{*}_{i}>> H^{k_{i}}(X;R) @>j^{*}_{i}>> H^{k_{i}}(U_{i};R) \rightarrow...\end{CD}.\end{gather*}

Using, once more, the fact that the Alexander-Spanier cohomology
coincides with the \v Cech cohomology on $X$ and $U_i$ and is,
therefore, representable, we have, for each $\alpha_{i}$,
$j^*_i(\alpha_i) = [f_i \circ j_i] = 0$. By exactness, there
is an element $\bar{\alpha_i} \in H^{k_i}(X,U_i;R)$ satisfying
$q^{*}_i(\bar{\alpha_i}) = \alpha_i$. The rest of the proof
goes in the same vein as in the case of a CW complex~\cite{CLOT}.
Namely, in view of the cup-product formula for the Alexander-Spanier
cohomology (see ~\cite{Sp}, pp. 315),
$$
H^k(X,U;R)\times H^l(X,V;R)\stackrel{\cup}{\longrightarrow} H^{k+l}(X,U\cup V;R),
$$
and the fact that $H^n(X,X;R)=0$, we obtain a contradiction.
\end{proof}

\end{thm}

\begin{ex}
Let $\Pi_2$ denote a Pontryagin surface constructed from the 2-sphere for the prime $2$. Then $\cat_g \Pi_2=2$.
\end{ex}
\begin{proof}
We recall that $\Pi_2$ is the inverse limit of a sequence~\cite{Ku},~\cite{Dr4},
$$
L_1 \stackrel{p_2}\longleftarrow L_2\stackrel{p_3}\longleftarrow L_3\stackrel{p_4}\longleftarrow\dots
$$
where $L_1=S^2$ with a fixed triangulation, each simplicial complex
$L_{i+1}$ is obtained from $L_i$ by  replacing every 2-simplex in
the barycentric subdivision by the (triangulated) M\"obius band, and
the bonding map $p_i$ sends this M\"obius band back to the
simplex. Thus each $L_i$, $i>1$, is a non-orientable surface. It is well known that there
is an $\alpha\in H^1(L_2; \Z_2)$ with $\alpha \smile \alpha \neq 0$ ~\cite{H}.
Since for $i>1$, $(p_i)_* : H_1(L_i;\Z_2) \rightarrow H_1(L_{i-1};\Z_2)$ is an isomorphism, this cup-product survives to the limit. Thus,  $\cuplength _{\Z_2} \Pi_2 > 1$. Since we also have
$\cat_g \Pi_2 \leq \dim \Pi_2 = 2$, we must have $\cat_g \Pi_2 = 2$.
\end{proof}
\begin{question}
Let $D_2$ be a Pontryagin surface constructed from the 2-disk. What is $\cat_g D_2$ ?
\end{question}
\begin{remark}
The above computation works for all Pontryagin surfaces $\Pi_p$ constructed from the $2$-sphere, where $p$ is any prime number.
The cup-length estimate in the case $p\ne 2$ requires cohomology with twisted coefficients. Another approach to obtaining a lower bound for the category of $\Pi_p$ is to use the category weight~\cite{CLOT}. Both approaches require substantial work.
\end{remark}

\end{document}